\documentclass[amscd,amssymb,verbatim,12pt]{amsart}

\usepackage{amssymb,latexsym}
\usepackage{enumerate}
\usepackage{graphicx}
\usepackage[usenames,dvipsnames]{pstricks}

\newtheorem{thm}{Theorem}[section]
\newtheorem{prop}[thm]{Proposition}
\newtheorem{cor}[thm]{Corollary}

\newtheorem{lemma}[thm]{Lemma}

\newtheorem*{therm}{Theorem}
\newtheorem*{thrmimpr}{Theorem \ref{impr}}
\newtheorem*{conjj}{Conjecture}
\newtheorem*{propp}{Proposition \ref{propcounter}}
\newtheorem*{dlh}{Varopoulos Isoperimetric Inequality}

\newtheorem{conjecture}[thm]{Conjecture}
\newcommand{\dsp}{\displaystyle}

\begin{document}

\title{On small separations in Cayley graphs}

\author{M. Giannoudovardi}
\address{Department of Mathematics,
University of Athens, Athens, Greece}
\email{marthag@math.uoa.gr}


\begin{abstract}
We present two results on expansion of Cayley graphs. The first result settles a conjecture made by DeVos and Mohar in \cite{one}. Specifically, we prove that for any positive constant $c$ there exists a finite connected subset $A$ of the Cayley graph of $\mathbb{Z}^2$ such that $\frac{|\partial A|}{|A|}< \frac{c}{depth(A)}$. 

This yields that there can be no universal bound for $\frac{|\partial A|depth(A)}{|A|}$ for subsets of either infinite or finite vertex transitive graphs.

Let $X=(V,E)$ be the Cayley graph of a finitely generated infinite group and $A\subset V$ finite such that $A\cup\partial A$ is connected. Our second result is that if $|A|> 16|\partial A|^2$ then $X$ has a ring-like structure.

\end{abstract}

\maketitle

\section{Introduction}
This paper presents results on expansion in vertex transitive graphs and an important subclass of vertex transitive graphs, infinite Cayley graphs. The expansion rate, or isoperimetric ratio, of vertex transitive graphs is known to have a strong relation to connectivity (local expansion) and growth (global expansion). We are mostly interested in expansion in Cayley graphs as this is related to the structure of infinite groups. We mention a few results.

Let $X=(V,E)$ be a connected vertex transitive graph and $A\subset V$ a finite subset of the vertex set. Babai and Szegedy proved in \cite{BabSz} that if $|A|\leqslant |V|/2$, then $|\partial A|/|A|\geqslant \frac{2}{2diam(A)+1}$. Motivated by this DeVos and Mohar conjectured in \cite{one} that $diam(A)$ may be replaced by a constant multiple of the depth of the set $A$. 
\begin{conjj}[DeVos, Mohar]
There exists a fixed constant $c>0$, so that in every connected vertex transitive graph, $X=(V,E)$, we have
$\frac{|\partial A|}{|A|}\geqslant \frac{C}{depth(A)}$
whenever $A\subset V$ is finite and $0<|A|\leqslant \frac{1}{2} |V|$.
\end{conjj}
Where $depth(A)=\sup\{d(u,V\smallsetminus A)\mid u\in A\}$.

In section \ref{counter} we prove the following inequality for the Cayley graph of $\mathbb{Z}^2$, thus providing a counter example to this conjecture.
\begin{prop}\label{propcounter}
Let $c>0$. There exists a finite subset, $A_c$, of $\mathbb{Z}^2$ so that, in the Cayley graph of $\mathbb{Z}^2$,
$$\frac{|\partial A_c|}{|A_c|}< \frac{c}{depth(A_c)}$$
\end{prop}
We remark that a corollary of this proposition is that for any positive constant $c$, there is a subset of the vertices of the Cayley graph of $\mathbb{Z}_n \times \mathbb {Z} _n$ for which the same inequality holds (where $n$ depends on $c$). Thus the conjecture is settled for both the infinite and the finite case.

In section \ref{imp}, we focus on local expansion in Cayley graphs of finitely generated infinite groups. Specifically, DeVos and Mohar proved in \cite{onepr}, a structure theorem that gives a characterization of vertex transitive graphs with small separations.
\begin{therm}[DeVos, Mohar]
Let $X=(V,E)$ be a vertex transitive graph, let $A\subset V$ be finite non-empty set with $|A|\leqslant\frac{|V|}{2}$ such that $A\cup\partial A$ is connected. Set $k=|\partial A|$ and assume that $diam(X)\geqslant 31(k+1)^2$. Then one of the following holds:
\begin{enumerate}[(i)]
\item $depth(A)\leqslant k$ and $|A|\leqslant 2k^3+k^2$
\item There exist integers $s$, $t$ with $st\leqslant \frac{k}{2}$ and a cyclic system $\vec{\sigma}$ on $X$ so that $X$ is $(s,t)$-ring-like, and there exists an interval $J$ of $\vec{\sigma}$ so that the set $Q=\bigcup\limits_{B\in J}B$ satisfies $A\subset Q$ and $|Q\smallsetminus A|\leqslant\frac{1}{2}k^3+k^2$.
\end{enumerate}
\end{therm}
They conjecture that the theorem should hold with a bound of the form $ck^2$ instead of $2k^3(1+o(1))$ in (\ref{(i)}). We prove a similar theorem, if $X$ is the Cayley graph of a finitely generated infinite group, with an improved bound for (\ref{(i)}) but without a bound on the ring like structure of $X$.
\begin{thm}\label{impr}
Let $X=(V,E)$ be a Cayley graph of a finitely generated infinite group $G$ with respect to a finite generating set, let $A\subset V$ finite such that $A\cup\partial A$ is connected and set $|\partial A|=k$. Then one of the following holds:
\begin{enumerate}[(i)]
\item $|A|\leqslant 16k^2$ and $depth(A)<4\sqrt2k$.
\item $|A|>16k^2$ and there exist positive integers $s,t$ and a cyclic system $\vec{\sigma}$ on $X$ so that $X$ is $(s,t)$-ring-like. Moreover there exists an interval, $J$, of $\vec{\sigma}$ so that the set $Q=\dsp\bigsqcup_{B\in J}B$ satisfies $A\subset Q$ and $|Q\smallsetminus A|\leqslant 2s^2t^2k+2stk$.
\end{enumerate}
\end{thm}
For that we use a result of Wilkie and Van Den Dries from \cite{wvdd} regarding the growth of a finitely generated group, the Varopoulos isoperimetric inequality \cite{var} and a result of DeVos and Mohar from \cite{onepr}.

\section{Preliminaries}
The notation introduced in this section will be used throughout this paper.\\
Let $X=(V,E)$ be a graph with vertex set $V$ and edge set $E$. If $A$ is a subgraph of $X$, then $|A|$ is the number of vertices in $A$. We say that $A$ is connected if for any vertices $u,v$ in $A$ there exists a finite sequence of consecutive edges (\textit{path}) in $A$ that starts from $u$ and ends in $v$. The \textit{distance of two vertices} $u,v\in V$ is zero if $u=v$ or the minimum number of edges that a path that joins $u$ to $v$ may have if $u\neq v$ and will be denoted by $d(u,v)$. Let $A$ be a non-empty subset of $V$. We say that $A$ is \textit{connected} if the corresponding subgraph of $X$ induced on the vertices of $A$ is connected. Moreover, for any $v\in V$, the \textit{distance of the vertex} $v$ \textit{from the set} $A$ is:
$$d(v,A)=\inf\{d(v,u)\mid u\in A\}$$
The \textit{boundary} of the set $A$ is the set:
$$\partial A = \{u\in V\smallsetminus A\mid \{u,v\}\in E, \text{ for some }v\in A\}$$
Clearly, an alternative expression for $\partial A$ is:
$$\partial A= \{u\in V\mid d(u,A)=1\}$$
The \textit{depth of the set} $A$ is the supremum, over all vertices in $A$, of their distance from $V\smallsetminus A$:
$$depth(A) = \sup\{d(u, V\smallsetminus A)\mid u\in A\}$$

The graph $X$ is \textit{vertex transitive} if for any $u,v\in V$, there exists an automorphism, $g$, of $X$ so that $g(u)=v$. It is evident that if $X$ is a locally finite, vertex transitive graph, then any two vertices have the same number of edges incident to them and this number is called the \textit{degree} of the graph $X$. Moreover, two balls of the same radius have the same, finite, number of vertices. Thus, for any positive integer $m$, we will denote the \textit{number of vertices in a ball} of radius $m$ in $X$ by $b(m)\in\mathbb{N}$.

Let $X=(V,E)$ be a connected vertex transitive graph. For any $A\subset V$ we set $X\smallsetminus A=(V\smallsetminus A,E\smallsetminus E(A))$, where $E(A)$ are the edges in $E$ that are contained in the subgraph of $X$ induced on the vertices of $A$. The \textit{number of ends} of $X$, $e(X)$, is the supremum over all compact subsets $A\subset V$ of the number of infinite connected components of $X\smallsetminus A$. It is a well known fact that $e(X)$ is $0,1,2$ or $\infty$ (Hopf \cite{hopf}, Halin \cite{hal}) as well as that $X$ has linear growth if and only if $e(G)=2$ (Imrich and Seifter \cite{imsei}).
We recall some definitions and results from \cite{onepr}, which we will use in section \ref{imp}.\\
A partition $\sigma$ of the set of vertices $V$ is a \textit{system of imprimitivity} if it is invariant under the action of the automorphism group of the graph $X$, i.e. if for any automorphism $g$ of the graph $X$ and $B\in\sigma$, we have that $g(B)\in\sigma$. The sets of the partition $\sigma$ are called \textit{blocks of imprimitivity}. 
If $A$ is a set, then a \textit{cyclic order} on $A$ is a symmetric relation $\sim$ such that the corresponding graph, $\tilde{A}$, is either a circuit or a bi-infinite path. The \textit{distance} of two elements $x,y\in A$ in the cyclic order is their distance in $\tilde{A}$ and is denoted by $\tilde{d}(x,y)$. An \textit{interval} in $\tilde{A}$ is a finite subset $\{x_1,x_2,\dots,x_k\}\subset A$ such that for any $1\leqslant i\leqslant k-1$, $x_i\sim x_{i+1}$. 
A \textit{cyclic system}, $\vec{\sigma}$ on the graph $X$ is a system of imprimitivity, $\sigma$, equipped with a cyclic order which is preserved by the automorphism group of the graph $X$. Let $s,t$ be positive integers, then $X$ is $(s,t)$\textit{-ring-like} if there exists a cyclic system $\vec{\sigma}$ on $X$ such that every block has $s$ elements and when $x,y$ are two adjacent vertices of $X$ then $\tilde{d}(x,y)\leqslant t$. If $X$ is $(s,t)$-ring-like with respect to a cyclic system $\vec{\sigma}$, then $X$ is $q$-\textit{cohesive}, for some $q\in\mathbb{N}$, if any two vertices of $X$ which are in the same or adjacent blocks can be joined by a path of length at most $q$.\\
The following two results regarding expansion in two ended vertex transitive graphs were proven by DeVos and Mohar in \cite{onepr}.
\begin{thm}[DeVos, Mohar]\label{theorem3.1}
Let $X$ be a connected vertex transitive graph with two ends. Then there exist integers $s,t$ and a cyclic system $\vec{\sigma}$ so that $X$ is $(s,t)$-ring-like and $2st$-cohesive with respect to $\vec{\sigma}$.
\end{thm}
\begin{lemma}[DeVos, Mohar]\label{lemma5.1}
Let $X=(V,E)$ be a vertex transitive graph which is $(s,t)$-ring-like and $2st$-cohesive with respect to a cyclic system $\vec{\sigma}$. Let $A\subset V$ and assume that $A\cup \partial A$ is connected, $|A|\leqslant \frac{1}{2}|V|$ and set $|\partial A|=k$. Then there exists an interval, $J$, of $\vec{\sigma}$ so that the set $Q=\dsp\bigsqcup_{B\in J}B$ satisfies $A\subset Q$ and $|Q\smallsetminus A|\leqslant 2s^2t^2k+2stk$.
\end{lemma}

Finally, let $X=(V,E)$ be the Cayley graph of a finitely generated infinite group $G$ with respect to a finite generating set $S$. Then $X$ is a connected locally finite, infinite vertex transitive graph. As the number of ends is a quasi isometry invariant, the number of ends of the group $G$, $e(G)$, is the number of ends of its Cayley graph with respect to any finite generating set, so $e(G)=e(X)$.
Furthermore, the following result of Wilkie and Van Den Dries in \cite{wvdd} describes explicitly the growth of locally finite Cayley graphs.
\begin{thm}[Wilkie, Van Den Dries]\label{wvd}
Let $X$ be a Cayley graph of a finitely generated group. Then only one of the following holds:
\begin{enumerate}
\item There exist $\alpha$, $\beta\geqslant 0$ so that for any $n\in\mathbb{N}$, $b(n)\leqslant \alpha n+\beta$
\item For any $n\in\mathbb{N}$, $b(n)\geqslant \frac{1}{2}(n+1)(n+2)$
\end{enumerate}
\end{thm}
To complete our reference in expansion in Cayley graphs we recall the Varopoulos isoperimetric inequality \cite{var}:
\begin{dlh}
Let $X=(V,E)$ be a Cayley graph of a finitely generated group. If $A$ is a non empty, finite subset of $V$ and $m$ is the minimum positive integer so that $b(m)\geqslant 2|A|$, then
$$|A|\leqslant 2m|\partial A|$$
\end{dlh}
 
\section{A counterexample}\label{counter}
Motivated by the work of Babai and Szegety in \cite{BabSz}, DeVos and Mohar made in \cite{one} the following conjecture concerning local expansion in vertex transitive graphs.

\begin{conjecture}[DeVos, Mohar]\label{conj}
There exists a fixed constant $c>0$, so that in every locally finite vertex transitive graph $X=(V,E)$, we have
$$\frac{|\partial A|}{|A|}\geqslant \frac{c}{depth(A)}$$
whenever $A\subset V$ is finite and $\displaystyle2|A|\leqslant |V|$.
\end{conjecture}
We will disprove this conjecture by constructing a counter example. Specifically, we will first show that for the graph of $\mathbb{Z}^2$ there exists no such constant.
\begin{propp}
Let $c>0$. There exists a finite subset, $A_c$, of $\mathbb{Z}^2$ so that in the Cayley graph of $\mathbb{Z}^2$:
$$\frac{|\partial A_c|}{|A_c|}< \frac{c}{depth(A_c)}$$
\end{propp}

\begin{proof}
For any $i$, $k\in\mathbb{Z}$ with $i,k>1$, let:
$$\begin{array}{r c l}
X_i(k)&=& \{(m,n)\in\mathbb{Z}^2\mid 0\leqslant m\leqslant ki \text{ , } 0\leqslant n\leqslant ki\}\\
Y_i(k) &=& \{(mi,ni)\in\mathbb{Z}^2\mid 1\leqslant m\leqslant k-1\text{ , }1\leqslant n\leqslant k-1\}\\
\end{array}$$
and
$$A_i(k) = X_i(k) \smallsetminus Y_i(k)$$
\begin{figure}[ht]
    \centering
        \includegraphics[scale=0.5]{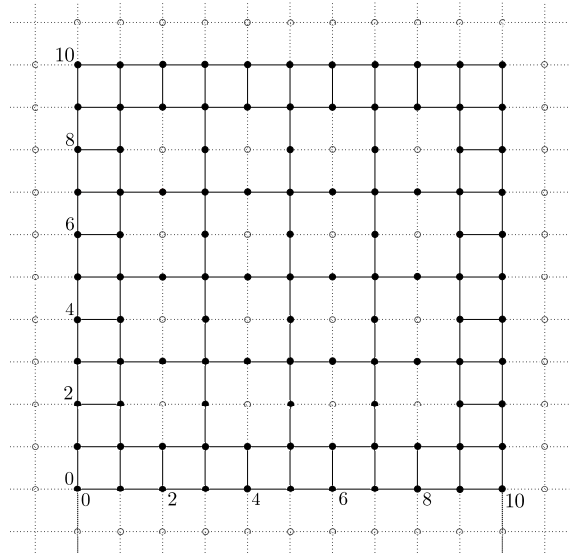}
       \caption{The set $A_2(5)$ in the graph of $\mathbb{Z}^2$}
\label{Fig:1}
\end{figure}\\

It is easy to check that, for any $i,k>1$, the set $A_i(k)$ is a finite subset of the vertices of the graph of $\mathbb{Z}^2$ and it satisfies the following:
\begin{enumerate}
\item $depth(A_i(k))\leqslant \frac{i}{2}$\label{oone}
\item $|A_i(k)| = (ki)^2 - (k-1)^2$\label{two}
\item $|\partial A_i(k)| = (k-1)^2 + 4(ki+1)$\label{three}
\end{enumerate}
Let 
$$a_i(k)=\displaystyle{\frac{|\partial A_i(k)|\cdot depth(A_i(k))}{|A_i(k)|}}$$
Then, by (\ref{oone}),(\ref{two}),(\ref{three}) we get that for any $i,k>1$:
$$
\begin{array}{r c l}
a_i(k)
 & \leqslant & \displaystyle{\frac{[(k-1)^2+4(ki+1)]i}{2[(ki)^2-(k-1)^2]}}\Rightarrow\\
  & & \\
a_i(k) & \leqslant & \displaystyle{\frac{ik^2+(4i^2-2i)k+5i}{2(i^2-1)k^2+4k-2}}\\
\end{array}
$$
Therefore,
$$\lim\limits_{i,k\to\infty}a_i(k)=0$$
Thus, for any $c>0$ there exist $i_c,k_c\in\mathbb{N}$ so that $a_{i_c}(k_c)<c$, i.e. the set $A_c=A_{i_c}(k_c)$ satisfies the inequality:
$$\displaystyle{\frac{|\partial A_c|\cdot depth(A_c)}{|A_c|}}<c$$
\end{proof}

The following corollary shows that there exists no such universal constant even if we restrict the conjecture to finite graphs.
\begin{cor}\label{corcounter}
Let $c>0$. There exists a finite graph $X=(V,E)$ and $A\subset V$, with $\displaystyle0<|A|\leqslant \frac{1}{2}|V|$ so that
$$\frac{|\partial A|}{|A|}< \frac{c}{depth(A)}$$
\end{cor}

\begin{proof}
From Proposition \ref{propcounter}, there exists a finite set $A=A_{i}(k)\subset \mathbb{Z}^2$ such that $\displaystyle\frac{|\partial A|}{|A|}< \frac{c}{depth(A)}$.
Let $X=(V,E)$ be the Cayley graph of $\mathbb{Z}_{n}\bigoplus \mathbb{Z}_{n}$, for $n=3ki+1$.
\begin{figure}[ht]
    \centering
        \includegraphics[scale=1]{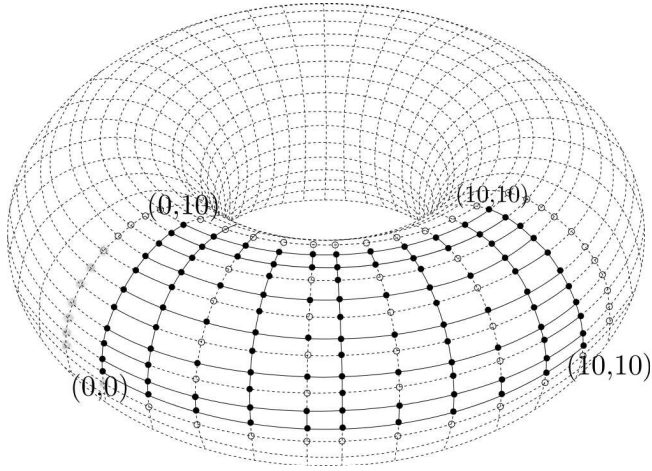}
       \caption{The set $A_2(5)$ in the graph of $\mathbb{Z}_{31}\oplus\mathbb{Z}_{31}$}
\label{Fig:2}
\end{figure}\\
Then $X$ is a finite vertex transitive graph, and $A\subset V$ with $\displaystyle0<|A|\leqslant \frac{1}{2}|V|$ and $\displaystyle\frac{|\partial A|}{|A|}< \frac{c}{depth(A)}$.
\end{proof}
 
\section{A rough structure theorem for infinite Cayley graphs}\label{imp}
In \cite {onepr}, DeVos and Mohar prove the following theorem for small separations in vertex transitive graphs:
\begin{thm}[DeVos, Mohar]
Let $X=(V,E)$ be a vertex transitive graph, and $A\subset V$ be finite and non empty set with $|A|\leqslant\frac{1}{2}|V|$ such that $A\cup\partial A$ is connected. Set $k=|\partial A|$ and assume that $diam(X)\geqslant 31(k+1)^2$. Then one of the following holds:
\begin{enumerate}[(i)]
\item $depth(A)\leqslant k$ and $|A|\leqslant 2k^3+k^2$\label{(i)}
\item There exist integers $s$, $t$ with $st\leqslant \frac{k}{2}$ and a cyclic system $\vec{\sigma}$ on $X$ so that $X$ is $(s,t)$-ring-like and there exists an interval $J$ of $\vec{\sigma}$ so that the set $Q=\bigcup\limits_{B\in J}B$ satisfies $A\subset Q$ and $|Q\smallsetminus A|\leqslant\frac{1}{2}k^3+k^2$.
\end{enumerate}
\end{thm}
Moreover, they conjecture that the theorem should hold with a bound of the form $ck^2$ instead of $2k^3(1+o(1))$ in (\ref{(i)}). We consider the case where $X$ is the Cayley graph of a finitely generated infinite group and prove the following structure theorem with an improved bound for (\ref{(i)}) but without a bound on the ring like structure of $X$.
\begin{thrmimpr}
Let $X=(V,E)$ be a Cayley graph of a finitely generated infinite group $G$ with respect to a finite generating set, let $A\subset V$ finite such that $A\cup\partial A$ is connected and set $|\partial A|=k$. Then one of the following holds:
\begin{enumerate}[(i)]
\item $|A|\leqslant 16k^2$ and $depth(A)<4\sqrt2k$.
\item $|A|>16k^2$ and there exist positive integers $s,t$ and a cyclic system $\vec{\sigma}$ on $X$ so that $X$ is $(s,t)$-ring-like. Moreover there exists an interval, $J$, of $\vec{\sigma}$ so that the set $Q=\dsp\bigsqcup_{B\in J}B$ satisfies $A\subset Q$ and $|Q\smallsetminus A|\leqslant 2s^2t^2k+2stk$.
\end{enumerate}
\end{thrmimpr}
The proof consists of two lemmas (Lemmas \ref{case1} and \ref{newcase2}), involving the number of ends of the group $G$. 

\begin{lemma}\label{case1}
Let $X=(V,E)$ be the Cayley graph of a finitely generated group $G$ with one or infinitely many ends, with respect to a finite generating set. If $A\subset V$ is finite and $|\partial A|=k$, then
$$|A|\leqslant 16k^2\quad\text{and}\quad depth (A)<4\sqrt2k$$
\end{lemma}

\begin{proof}
From Theorem \ref{wvd} we have that, for any $n\in\mathbb{N}$,
$$b(n)\geqslant \frac{1}{2}(n+1)(n+2)$$
Therefore, $b(2\sqrt{|A|})>2|A|$ and if $m$ is the minimum positive integer so that $b(m)\geqslant 2|A|$, then $m\leqslant 2\sqrt{|A|}$. So, from the Varopoulos isoperimetric inequality we derive that 
$$|A|\leqslant 16k^2$$
Also, there exists $x\in A$ such that $B(x,depth(A)-1)\subset A$, hence 
$$\left.\begin{tabular}{r c l}
$b(depth(A)-1)$&$\leqslant$ &$|A|<16k^2$\\
$b(depth(A)-1)$&$>$&$\displaystyle\frac{(depth(A))^2}{2}$\\
\end{tabular}\right\} \Rightarrow depth(A)<4\sqrt2k$$
\end{proof}
\begin{lemma}\label{newcase2}
Let $X=(V,E)$ be the Cayley graph of a finitely generated infinite group $G$ with respect to a finite generating set, let $A\subset V$ finite such that $A\cup\partial A$ is connected and set $|\partial A|=k$. If $|A|>16k^2$ then there exist positive integers $s,t$ and a cyclic system, $\vec{\sigma}$, on $X$ so that $X$ is $(s,t)$-ring-like. Moreover, there exists an interval, $J$, of $\vec{\sigma}$ so that the set $Q=\dsp\bigsqcup_{B\in J}B$ satisfies $A\subset Q$ and $|Q\smallsetminus A|\leqslant 2s^2t^2k+2stk$.
\end{lemma}
\begin{proof}
Since $|A|>16k^2$ from Lemma \ref{case1} we get that $X$ is two ended. So, from Theorem \ref{theorem3.1}, there exist integers $s,t$ and a cyclic system $\vec{\sigma}$ so that $X$ is $(s,t)$-ring-like and $2st$-cohesive with respect to $\vec{\sigma}$. Thus, from Lemma \ref{lemma5.1}, there exists an interval $J$ of $\vec{\sigma}$ so that the set $Q=\dsp\bigsqcup_{B\in J}B$ satisfies $A\subset Q$ and $|Q\smallsetminus A|\leqslant 2s^2t^2k+2stk$.
\end{proof}

\end{document}